\documentclass[preprint,12pt]{elsarticle}

\usepackage{amssymb}
\usepackage{amsthm}
\usepackage{amsmath}
\usepackage{epic}
\usepackage{CJK}
\usepackage{setspace}
\newtheorem{thm}{Theorem}[section]
\newtheorem{cor}[thm]{Corollary}
\newtheorem{lem}[thm]{Lemma}

\journal{}

\begin{document}
\begin{spacing}{1.15}
\begin{frontmatter}
\title{\textbf{On the spectral radius of nonregular uniform hypergraphs}}

\author[label1,label2]{Jiang Zhou}\ead{zhoujiang04113112@163.com}
\author[label3]{Lizhu Sun}
\author[label1]{Changjiang Bu}


\address{
\address[label1]{College of Science, Harbin Engineering University, Harbin 150001, PR China}
\address[label2]{College of Computer Science and Technology, Harbin Engineering University, Harbin 150001, PR China}
\address[label3]{School of Science, Harbin Institute of Technology, Harbin 150001, PR China}

}

\begin{abstract}
Let $G$ be a connected uniform hypergraphs with maximum degree $\Delta$, spectral radius $\lambda$ and minimum H-eigenvalue $\mu$. In this paper, we give some lower bounds for $\Delta-\lambda$, which extend the result of [S.M. Cioab\u{a}, D.A. Gregory, V. Nikiforov, Extreme eigenvalues of nonregular graphs, J. Combin. Theory, Ser. B 97 (2007) 483-486] to hypergraphs. Applying these bounds, we also obtain a lower bound for $\Delta+\mu$.
\end{abstract}

\begin{keyword}
Spectral radius, Hypergraph, Adjacency tensor, Eigenvalue\\
\emph{AMS classification:} 05C65, 05C50, 15A69, 15A18
\end{keyword}
\end{frontmatter}

\section{Introduction}
Spectral graph theory has been widely studied, and has many important applications in combinatorics, computer science, physics and so on. It is natural to generalize spectral theory to hypergraphs. Recently, there have been many attempts to develop spectral hypergraph theory based on eigenvalues of tensors [2,5-7,10-13,18,20].

An order $m$ dimension $n$ tensor $\mathcal{A}=(a_{i_1i_2\ldots i_m})$ is a multidimensional array with $n^m$ entries ($i_j\in\{1,\ldots,n\},j=1,\ldots,m$). $\mathcal{A}$ is called \textit{symmetric} if $a_{i_1i_2\cdots i_k}=a_{i_{\sigma(1)}i_{\sigma(2)}\cdots i_{\sigma(k)}}$ for any permutation $\sigma$ on $\{1,\ldots,k\}$. In 2005, the concept of eigenvalues of tensors was defined by Qi \cite{Qi 2005} and Lim \cite{Lim 2005}. For $\mathcal{A}=(a_{i_1i_2\cdots i_m})\in\mathbb{C}^{n\times n\times\cdots\times n}$ and $x=\left({x_1 ,\ldots ,x_n}\right)^\mathrm{T}\in\mathbb{C}^n,$ $\mathcal{A}x^{m-1}$ is a vector in $\mathbb{C}^n$ whose $i$-th component is
\begin{align*}
(\mathcal{A}x^{m-1})_i=\sum\limits_{i_2,\ldots,i_m=1}^na_{ii_2\cdots i_m}x_{i_2}\cdots x_{i_m}.
\end{align*}
A number $\lambda\in\mathbb{C}$ is called an \textit{eigenvalue} of $\mathcal{A}$, if there exists a nonzero vector $x\in\mathbb{C}^n$ such that $\mathcal{A}x^{m-1}=\lambda x^{[m-1]}$, where $x^{\left[ {m - 1} \right]}  = \left( {x_1^{m - 1},\ldots,x_n^{m - 1} } \right)^\top$. In this case, $x$ is an \textit{eigenvector} of $\mathcal{A}$ associated with $\lambda$. If $\lambda$ is a real eigenvalue with a real eigenvector, then $\lambda$ is called an \textit{H-eigenvalue} of $\mathcal{A}$. The \textit{spectral radius} of $\mathcal{A}$ is defined as $\rho(\mathcal{A})=\max\{|\lambda|:\lambda\in\sigma(\mathcal{A})\}$, where $\sigma(\mathcal{A})$ is the set of all eigenvalues of $\mathcal{A}$.

Let $V(G)$ and $E(G)$ denote the vertex set and edge set of a hypergraph $G$, respectively. If each edge of $G$ contains exactly $k$ distinct vertices, then $G$ is called $k$-\textit{uniform}. We will use the term \textit{$k$-graph} in place of $k$-uniform hypergraph. Clearly, a $2$-graph is a simple undirected graph. The \textit{degree} of a vertex $u$ of $G$ is the number of edges containing $u$. If all vertices of $G$ have the same degree, then $G$ is called \textit{regular}. In a $k$-graph $G$, a \textit{path of length $l$} is defined to be an alternating sequence of vertices and edges $u_1,e_1,u_2,\ldots,u_l,e_l,u_{l+1}$, where $u_1,\ldots,u_{l+1}$ are distinct vertices of $G$, $e_1,\ldots,e_l$ are distinct edges of $G$ and $u_i,u_{i+1}\in e_i$ for $i=1,\ldots,l$. If there exists a path between any two vertices of $G$, then $G$ is called \textit{connected}. The distance between two vertices is the length of the shortest path connecting them. The \textit{diameter} of a connected $k$-graph $G$ is the maximum distance among all vertices of $G$.

The \textit{adjacency tensor} \cite{Cooper12} of a $k$-graph $G$ with $n$ vertices, denoted by $\mathcal{A}_G$, is an order $k$ dimension $n$ symmetric tensor with entries
\begin{eqnarray*}
a_{i_1i_2\cdots i_k}=\begin{cases}\frac{1}{(k-1)!}~~~~~~~\mbox{if}~i_1i_2\cdots i_k\in E(G),\\
0~~~~~~~~~~~~~\mbox{otherwise}.\end{cases}
\end{eqnarray*}
Eigenvalues (H-eigenvalues) of $\mathcal{A}_G$ are called eigenvalues (H-eigenvalues) of $G$. Let $\mu(G)$ denote the minimum H-eigenvalue of $G$. The spectral radius of $\mathcal{A}_G$ is called the spectral radius of $G$, denoted by $\rho(G)$. When $k=2$, $\mathcal{A}_G$ is the adjacency matrix of $2$-graph $G$.

For a connected $k$-graph $G$ with maximum degree $\Delta$, Cooper and Dutle \cite{Cooper12} proved that $\rho(G)\leqslant\Delta$, with equality if and only if $G$ is regular. It is natural to consider how small $\Delta-\rho(G)$ can be for nonregular $k$-graphs. When $G$ is nonregular $2$-graph, some lower bounds on $\Delta-\rho(G)$ are given in \cite{Cioaba0,Cioaba,Liu,Stevanovic,Zhang}. These bounds are also lower bounds of $\Delta+\mu(G)$ because $\Delta+\mu(G)\geqslant\Delta-\rho(G)$.

Alon and Sudakov \cite{Alon} proved that if $G$ is a connected nonbipartite (possibly regular) $2$-graph with $n$ vertices, then
\begin{eqnarray*}
\Delta+\mu(G)\geqslant\frac{1}{n(D+1)},
\end{eqnarray*}
where $D$ is the diameter of $G$. For a connected nonregular $2$-graph $G$ with $n$ vertices and $m$ edges, Cioab\u{a}, Gregory and Nikiforov \cite{Cioaba} obtain the following bound
\begin{align}
\Delta-\rho(G)>\frac{n\Delta-2m}{n(D(n\Delta-2m)+1)}.\tag{1.1}
\end{align}
In this paper, we give some lower bounds on $\Delta-\rho(G)$ for a connected nonregular $k$-graph $G$, which extend the inequality (1.1) to hypergraphs. For a connected non-odd-bipartite $k$-graph $G$, we show that $\Delta+\mu(G)>\frac{1}{3n}$ when $k\geqslant4$ is even.
\section{Preliminaries}
In this section, we collect some helpful lemmas.
\begin{lem}\textup{\cite{ChangAn}}\label{lem2.1}
Let $a_1,\ldots,a_n$ be nonnegative real numbers. Then
\begin{eqnarray*}
\frac{a_1+\cdots+a_n}{n}-(a_1\cdots a_n)^{\frac{1}{n}}\geqslant\frac{1}{n(n-1)}\sum_{1\leqslant i<j\leqslant n}(\sqrt{a_i}-\sqrt{a_j})^2.
\end{eqnarray*}
\end{lem}

\begin{lem}\label{lem2.2}
Let $a,b,y_1,y_2$ be positive numbers. Then $a(y_1-y_2)^2+by_2^2\geqslant\frac{ab}{a+b}y_1^2$.
\end{lem}
\begin{proof}
By computation, we have
\begin{eqnarray*}
a(y_1-y_2)^2+by_2^2=(a+b)(y_2-\frac{ay_1}{a+b})^2+\frac{ab}{a+b}y_1^2\geqslant\frac{ab}{a+b}y_1^2.
\end{eqnarray*}
\end{proof}
A hypergraph $G$ is called \textit{$f$-edge-connected} if $G-U$ is connected for any edge subset $U\subseteq E(G)$ satisfying $|U|<f$. Two paths are called \textit{edge-disjoint} if they does not have a common edge. The following is the Menger's theorem for hypergraphs.
\begin{lem}\textup{\cite{Zykov}}\label{lem2.3}
A hypergraph $G$ is $f$-edge-connected if and only if there are $f$ mutual edge-disjoint paths between each pair of vertices.
\end{lem}
Let $\lambda$ be an eigenvalue of a $k$-graph $G$ with eigenvector $x$. Since $\mathcal{A}_Gx^{k-1}=\lambda x^{[k-1]}$, we know that $cx$ is also an eigenvector of $\lambda$ for any nonzero constant $c$. So we can choose $x$ such that $\sum_{i=1}^nx_i^k=1$. In this case, we have \cite{Cooper12,LiShaoQi}
\begin{align}
\lambda=x^\top(\mathcal{A}_Gx^{k-1})=k\sum_{e\in E(G)}x^e,\tag{2.1}
\end{align}
where $x^e=x_{i_1}x_{i_2}\cdots x_{i_k}$, $i_1,i_2,\ldots,i_k$ are $k$ vertices in the edge $e$. From the proof of [14, Theorem 5], we obtain the following lemma.
\begin{lem}\label{lem2.4}
Let $G$ be a connected $k$-graph with $n$ vertices and $k$ is even. Then
\begin{eqnarray*}
\mu(G)=\min_{x\in\mathbb{R}^n,\sum_{i=1}^nx_i^k=1}x^\top(\mathcal{A}_Gx^{k-1}).
\end{eqnarray*}
\end{lem}
\section{Main results}
Let $G$ be a connected $k$-graph with $n$ vertices and $m$ edges, and let $\Delta$ and $D$ be the maximum degree and the diameter of $G$, respectively. We give some lower bounds for $\Delta-\rho(G)$ as follows.
\begin{thm}
Let $G$ be a nonregular connected $k$-graph. Then
\begin{eqnarray*}
\Delta-\rho(G)>\frac{k(n\Delta-km)}{n\left(2(k-1)D(n\Delta-km)+k\right)}.
\end{eqnarray*}
Moreover, the following statements hold:\\
(1) If $k\geqslant5$ and $G$ is $f$-edge-connected, then
\begin{eqnarray*}
\Delta-\rho(G)>\frac{fk(n\Delta-km)}{n[2(k-1)(n\Delta-km)+fk]}.
\end{eqnarray*}\\
(2) If $k=4$ and $G$ is $f$-edge-connected, then
\begin{eqnarray*}
\Delta-\rho(G)>\frac{f(n\Delta-4m)}{n[2(n\Delta-4m)+f]}.
\end{eqnarray*}
\end{thm}
\begin{proof}
It is known \cite{Qi14} that $\rho(G)$ is an eigenvalue of $G$ with a positive eigenvector $x$. We choose $x$ such that $\sum_{i=1}^nx_i^k=1$. Let $u,v$ be two vertices such that $x_u=\max_{i\in V(G)}x_i$ and $x_v=\min_{i\in V(G)}x_i$. By $\sum_{i=1}^nx_i^k=1$, we get $x_u^k>\frac{1}{n}>x_v^k$ ($G$ is regular when $x_u^k=x_v^k=\frac{1}{n}$). By (2.1), we have
\begin{eqnarray*}
\Delta-\rho(G)&=&\Delta\sum_{i=1}^nx_i^k-k\sum_{e\in E(G)}x^e=\sum_{i=1}^n(\Delta-d_i)x_i^k+\sum_{i=1}^nd_ix_i^k-k\sum_{e\in E(G)}x^e\\
&=&\sum_{i=1}^n(\Delta-d_i)x_i^k+\sum_{e=i_1\cdots i_k\in E(G)}(x_{i_1}^k+\cdots+x_{i_k}^k-kx^e),
\end{eqnarray*}
where $d_i$ is the degree of the vertex $i$. Since $x_u^k>x_v^k>0$, by Lemma \ref{lem2.1}, we have
\begin{align}
\Delta-\rho(G)>(n\Delta-km)x_v^k+\frac{1}{k-1}\sum_{i,j\in e\in E(G)}(x_i^{\frac{k}{2}}-x_j^{\frac{k}{2}})^2.\tag{3.1}
\end{align}
Let $P:u=u_0,e_1,u_1,\ldots,u_{r-1},e_r,u_r=v$ be a path from $u$ to $v$, where $e_i$ is an edge containing vertices $u_{i-1}$ and $u_i$. Then
\begin{eqnarray*}
\sum_{i,j\in e\in E(P)}(x_i^{\frac{k}{2}}-x_j^{\frac{k}{2}})^2\geqslant\sum_{i=0}^{r-1}(x_{u_i}^{\frac{k}{2}}-x_{u_{i+1}}^{\frac{k}{2}})^2+
\sum_{v_i\in e_i\backslash\{u_{i-1},u_i\}}\sum_{i=0}^{r-1}[(x_{u_i}^{\frac{k}{2}}-x_{v_{i+1}}^{\frac{k}{2}})^2+(x_{v_{i+1}}^{\frac{k}{2}}-x_{u_{i+1}}^{\frac{k}{2}})^2].
\end{eqnarray*}
It follows from the Cauchy-Schwarz inequality that
\begin{eqnarray*}
\sum_{i,j\in e\in E(P)}(x_i^{\frac{k}{2}}-x_j^{\frac{k}{2}})^2&\geqslant&\frac{1}{r}\left(\sum_{i=0}^{r-1}(x_{u_i}^{\frac{k}{2}}-x_{u_{i+1}}^{\frac{k}{2}})\right)^2+\sum_{v_i\in e_i\backslash\{u_{i-1},u_i\}}\frac{1}{2r}\left(\sum_{i=0}^{r-1}(x_{u_i}^{\frac{k}{2}}-x_{u_{i+1}}^{\frac{k}{2}})\right)^2\\
&=&\frac{1}{r}(x_u^{\frac{k}{2}}-x_v^{\frac{k}{2}})^2+\frac{(k-2)^r}{2r}(x_u^{\frac{k}{2}}-x_v^{\frac{k}{2}})^2=\frac{2+(k-2)^r}{2r}(x_u^{\frac{k}{2}}-x_v^{\frac{k}{2}})^2.
\end{eqnarray*}
So
\begin{align}
\sum_{i,j\in e\in E(P)}(x_i^{\frac{k}{2}}-x_j^{\frac{k}{2}})^2\geqslant\frac{2+(k-2)^r}{2r}(x_u^{\frac{k}{2}}-x_v^{\frac{k}{2}})^2\geqslant\frac{k}{2r}(x_u^{\frac{k}{2}}-x_v^{\frac{k}{2}})^2.\tag{3.2}
\end{align}
There is a shortest path from $u$ to $v$ whose length does not exceed the diameter $D$. By (3.1) and (3.2), we have
\begin{eqnarray*}
\Delta-\rho(G)&>&(n\Delta-km)x_v^k+\frac{k}{2(k-1)D}(x_u^{\frac{k}{2}}-x_v^{\frac{k}{2}})^2.
\end{eqnarray*}
The right side of the above inequality is a quadratic function of $x_v^{\frac{k}{2}}$. By Lemma \ref{lem2.2}, we get
\begin{eqnarray*}
\Delta-\rho(G)>\frac{k(n\Delta-km)}{2(k-1)D(n\Delta-km)+k}x_u^k.
\end{eqnarray*}
Since $x_u^k>\frac{1}{n}$, we get
\begin{eqnarray*}
\Delta-\rho(G)>\frac{k(n\Delta-km)}{n\left(2(k-1)D(n\Delta-km)+k\right)}.
\end{eqnarray*}
Next we consider the cases of $k\geqslant5$ and $k=4$.

\textbf{Case 1. $k\geqslant5$ and $G$ is $f$-edge-connected.} In this case, $\frac{2+(k-2)^r}{2r}\geqslant\frac{k}{2}$. By (3.2), we have
\begin{align}
\sum_{i,j\in e\in E(P)}(x_i^{\frac{k}{2}}-x_j^{\frac{k}{2}})^2\geqslant\frac{k}{2}(x_u^{\frac{k}{2}}-x_v^{\frac{k}{2}})^2.\tag{3.3}
\end{align}
Since $G$ is $f$-edge-connected, by Lemma \ref{lem2.3}, there are $f$ mutual edge-disjoint paths between $u$ and $v$. By (3.1) and (3.3), we have
\begin{eqnarray*}
\Delta-\rho(G)>(n\Delta-km)x_v^k+\frac{fk}{2(k-1)}(x_u^{\frac{k}{2}}-x_v^{\frac{k}{2}})^2.
\end{eqnarray*}
By Lemma \ref{lem2.2}, we get
\begin{eqnarray*}
\Delta-\rho(G)>\frac{fk(n\Delta-km)}{2(k-1)(n\Delta-km)+fk}x_u^k>\frac{fk(n\Delta-km)}{n[2(k-1)(n\Delta-km)+fk]}.
\end{eqnarray*}

\textbf{Case 2. $k=4$ and $G$ is $f$-edge-connected.} In this case, $\frac{2+(k-2)^r}{2r}\geqslant\frac{3}{2}$. By (3.2), we have
\begin{align}
\sum_{i,j\in e\in E(P)}(x_i^2-x_j^2)^2\geqslant\frac{3}{2}(x_u^2-x_v^2)^2.\tag{3.4}
\end{align}
Since $G$ is $f$-edge-connected, by Lemma \ref{lem2.3}, there are $f$ mutual edge-disjoint paths between $u$ and $v$. By (3.1) and (3.4), we have
\begin{eqnarray*}
\Delta-\rho(G)>(n\Delta-4m)x_v^4+\frac{f}{2}(x_u^2-x_v^2)^2.
\end{eqnarray*}
By Lemma \ref{lem2.2}, we get
\begin{eqnarray*}
\Delta-\rho(G)>\frac{f(n\Delta-4m)}{2(n\Delta-4m)+f}x_u^4>\frac{f(n\Delta-4m)}{n[2(n\Delta-4m)+f]}.
\end{eqnarray*}
\end{proof}
\noindent
\textbf{Remark 1.} Take $k=2$ in Theorem 3.1, we can obtain the inequality (1.1).

\begin{cor}
Let $G$ be a nonregular connected $k$-graph. If $k\geqslant4$, then
\begin{eqnarray*}
\Delta-\rho(G)>\frac{1}{3n}.
\end{eqnarray*}
\end{cor}
\begin{proof}
All connected $k$-graphs are $1$-edge-connected. If $k=4$, then by Theorem 3.1, we have
\begin{eqnarray*}
\Delta-\rho(G)>\frac{n\Delta-4m}{n[2(n\Delta-4m)+1]}\geqslant\frac{1}{3n}.
\end{eqnarray*}
If $k\geqslant5$, then by Theorem 3.1, we have
\begin{eqnarray*}
\Delta-\rho(G)>\frac{k(n\Delta-km)}{n[2(k-1)(n\Delta-km)+k]}>\frac{1}{3n}.
\end{eqnarray*}
\end{proof}

A $k$-graph $G$ is called \textit{odd-bipartite}, if there exists a proper subset $V_1$ of $V(G)$ such that each edge of $G$ contains exactly odd number of vertices in $V_1$ \cite{Zhou}. Clearly, odd-bipartite $2$-graphs are ordinary bipartite graphs. For a connected nonbipartite graph $G$ with $n$ vertices, Alon and Sudakov \cite{Alon} proved that $\Delta+\mu(G)\geqslant\frac{1}{n(D+1)}$, where $\Delta$ is the maximum degree of $G$, $D$ is the diameter of $G$. We consider similar problem for connected non-odd-bipartite $k$-graphs as follows.
\begin{thm}
Let $G$ be a connected non-odd-bipartite $k$-graph and $k\geqslant4$ is even. Then
\begin{eqnarray*}
\Delta+\mu(G)>\frac{1}{3n}.
\end{eqnarray*}
\end{thm}
\begin{proof}
It is known that $0$ is an H-eigenvalue of $G$ when $k\geqslant3$ \cite{Qi14}. So $\mu(G)\leqslant0$. If $\mu(G)=0$, then $\Delta+\mu(G)>\frac{1}{3n}$. If $G$ is nonregular, then by Corollary 3.2, we have $\Delta+\mu(G)\geqslant\Delta-\rho(G)>\frac{1}{3n}$. So we suppose that $G$ is regular and $\mu(G)<0$.

Let $x$ be a real eigenvector of $\mu(G)$ satisfying $\sum_{i=1}^nx_i^k=1$, and let $V_1=\{i:x_i<0\}$ and $V_2=\{i:x_i\geqslant0\}$. Since $\mu(G)<0$, by (2.1), $V_1$ is nonempty. If $V_2$ is empty, then $-x$ is a positive eigenvector of $\mu(G)$, a contradiction to the fact that $\rho(G)$ is the unique H-eigenvalue with a positive eigenvector (see \cite{Qi14}). Hence $V_1$ and $V_2$ are nonempty. Since $G$ is not odd-bipartite, it has an edge $f$ such that $|f\cap V_1|$ is even. Let $H=G-f$ be the $k$-graph obtained from $G$ by deleting the edge $f$. Since $G$ is regular and non-odd-bipartite, $H$ is nonregular and $|E(H)|\geqslant2$. By (2.1) and Lemma \ref{lem2.4}, we have
\begin{eqnarray*}
\mu(G)=k\sum_{e\in E(G)}x^e\geqslant k\sum_{e\in E(H)}x^e\geqslant\mu(H).
\end{eqnarray*}
Hence
\begin{eqnarray*}
\Delta+\mu(G)\geqslant\Delta+\mu(H)\geqslant\Delta-\rho(H).
\end{eqnarray*}
If $H$ is connected, then by Corollary 3.2, we have $\Delta+\mu(G)\geqslant\Delta-\rho(H)>\frac{1}{3n}$. If $H$ is disconnected, then there exists a component $H_1$ of $H$ such that $\rho(H)=\rho(H_1)$. Since $|E(H)|\geqslant2$ and $G$ is a connected regular $k$-graph, $H_1$ is nonregular. By Corollary 3.2, we have $\Delta+\mu(G)\geqslant\Delta-\rho(H)=\Delta-\rho(H_1)>\frac{1}{3n}$.
\end{proof}

\noindent
\textbf{Remark 2.} If $G$ is a connected odd-bipartite $k$-graph and $k$ is even, then $\mu(G)=-\rho(G)$ \cite{ShaoShanWu}. In this case, $\Delta+\mu(G)=\Delta-\rho(G)$.

\vspace{3mm}
\noindent
\textbf{Acknowledgements.}

\vspace{3mm}
This work is supported by the National Natural Science Foundation of China (No. 11371109 and No. 11426075), the Natural Science Foundation of the Heilongjiang Province (No. QC2014C001) and the Fundamental Research Funds for the Central Universities.

\end{spacing}
\end{document}